\documentclass{article}
\usepackage{hyperref}

\usepackage{comment}

\usepackage[applemac]{inputenc}   %
\usepackage[T1]{fontenc}
\usepackage[english]{babel}

\usepackage[pdftex]{graphicx}           
\usepackage{setspace}                   
\usepackage{indentfirst}                
\usepackage{enumerate}  
\usepackage{bbm}

\usepackage{amsmath}
\usepackage{amsthm}
\usepackage{amssymb}
\usepackage[fixlanguage]{babelbib}
\theoremstyle{definition}
\newtheorem{teo}{Theorem}
\newtheorem{lema}{Lemma}
\newtheorem{cor}{Corollary}

\newtheorem{remark}{Remark}

\def\Nset{{\mathbb{N} } }
\def\E{{\mathbb{E} } }

\def\P{\mathbb{P}}

\def\p{{\bf p}}
\def\q{{\bf q}}
\def\Q{{\cal Q}}
\def\nn{  {\nonumber}  }

\title{Large Deviation inequalities  for sums of positive correlated variables with  clustering }
\author{ 
        Miguel Abadi 
        \thanks{Instituto de Matemática, Estatística e Ciencia da Computação,, Universidade de São Paulo, São Paulo, Brasil }}
\date{\empty}

\begin{document}

\maketitle
\tableofcontents
\begin{abstract}
Large deviation inequalities for ergodic sums is an important subject  since the seminal contribution of Bernstein
for independent random variables with finite variances,  followed by  the Chernoff's method and the Hoefding result for independent bounded variables.
Very few results appears in the literature for the non independent case.
Here we consider the,  barely treated in the literature, case of positively correlated   Bernoulli variables.
This case represents the appearance in clusters of a certain fixed phenomena in the overlying stochastic process.
Under a very mild condition we prove several upper  deviation inequalities.
The results follow by a spectral decomposition of an appropriated recursive operator.
We illustrate with examples.
\end{abstract}

\section{Introduction}
In 1924, in his seminal paper, \emph{} Bernstein proved the following  deviation inequality. 
Let $X_0,\dots,X_{n-1}$ be a sequence of independent random variables with
\[
\mathbb{E}[X_i] = 0, 
\qquad 
|X_i| \le M \quad \text{a.s.}, 
\qquad 
\mathrm{Var}(X_i) = \sigma_i^2.
\]
Set
\[
S_n = \sum_{i=9}^{n-1} X_i, 
\qquad 
\sigma^2 = \sum_{i=0}^{n-1} \sigma_i^2.
\]
Then, for all  $\epsilon > 0$:
\[
\Pr\!\left( S_n \ge \epsilon \right) 
\;\le\; 
\exp\!\left( -\frac{\epsilon^2}{2  [\sigma^2 + \tfrac{M}{3}  \epsilon]} \right).
\]
Its application on Bernoulli($p$) variables reads as follows. Set $N_n=S_n$, then
\[
\Pr\!\left( N_n \ge (p+\epsilon) n \right) 
\;\le\; 
\exp\!\left( -\frac{\epsilon^2  n}{2  [p(1-p) +  \epsilon/3]}   \right).
\]
We use the letter $N_n$ to emphasize that in this case $S_n$ is the counting of successes along the sample $X_0,\dots, X_{n-1}.$  
The expression \emph{large deviations} stands to mean deviations of the order of the mean of $N_n$.
The interpretation of the above inequality is the following. One can identify two different regimes:
(1) For \emph{mild} large deviations, namely $\epsilon/3 \ll p(1-p)$, the bound becomes
        \[
    \Pr(N_n \ge (p+\epsilon)n  ) \;\lesssim\; \exp\!\left(-\frac{ \epsilon^2 n}{ 2p(1-p) }\right),
    \]
    that means a \emph{sub-Gaussian} tail for $N_n$.
   (2) For \emph{very} large deviations,     
    \[
    \Pr(N_n \ge (p+\epsilon)n ) \;\lesssim\; \exp\!\left(-\frac{ 3\epsilon n }{2}\right) , 
    \]
   that means  \emph{sub-exponential} tail.

Since then, many works were devoted to extend this inequality to various setting.
Still  in the case of independent summands, some of the most important contributions are those of Cramer, Feller, Chernoff , Okamoto , Bennett,  Hoeffding, 
which provide inequalities  for uniformly bounded, individually bounded, bounded differences, bounded moments, etc \cite{bennett, chernoff, cramer, feller, hoeffding,  okamoto}.

\vskip0.15cm
Yet a major challenge is to extend these results to dependent summands, even in the Bernoulli case.
Assuming extra condition,  several results appear in the literature, using different methods.
For instance, Janson proved an inequality for $m$-dependent variables, 
Merlevède et al. did the same for $\alpha$-mixing processes, 
and  Lezaud and  Paulin have results  for ergodic Markov chains \cite{janson, merlevede, lezaud, paulin}.
There are also further extensions for martingales by Azuma (Azuma--Hoeffding  inequality),
negatively correlated $\phi$-mixing variables  by Lu et al. 
and a recent paper by Mikosch   and Rodionov  for subexponential variables
\cite{azuma, lu, mikosch}.
As fundamental  survey we refer to the nice papers of  Janson \cite{janson2} and the references therein. \\

The goal of this paper is to present  several inequalities for the large deviation of a sum of indicator random variables which extends previous situations in two directions. 
(1) There is no asymptotic independence among distant variables. 
(2) The method  is designed  specially for, but not limited to,  positive correlated variables.
That is,  the occurrence  of an observable has  tendency to appear in clusters rather than  isolated. 
 Namely, the occurrence of  a second one is favoured by the occurrence of a previous one jn the immediate past.

The motivation comes from treating practical problems in which a particular phenomena rarely occurs, say with small probability $p$, but once it happens, immediate repetitions occurs with 
medium probability, say $\rho$.  For instance, large values of  a share on the stock exchange market,  high or low values of temperatures in a given region, 
large values of contaminated people of a given disease, number of days a strike last, number of days of replicas of an earthquake, etc.
For intuition the reader could have  in mind $\rho >  p$.
The size of the cluster of occurrences, even  assumed to have finite expectation,
is assumed to be unbounded. 
This fact is the source of a major difficulty  in the proof and represents a major contribution   with respect to previous results that assume the random sequence  being  bounded in some way.
Heuristically, the sample can be described as follows.
The  sum corresponds to the number of occurrences of an observable $A$ with tendency to appear in clusters, 
and  the mean number of observations  must respect a law of large numbers (or the Ergodic Theorem).
To that, if the occurrence of $A$ comes in clusters, 
the bigger the clusters, the larger the distance between them.
This situation makes bigger fluctuations and more subtle deviations.

Further, our results apply not only to clusters of  \emph{immediately consecutive}  repetitions. 
The observable  may be periodic of fixed period  strictly larger than one,
or even could be defined by the  occurrences of the observable within a distance up to a maximum  value $q$. 
For instance, in the periodic case, consider an indefinite strike   in a factory that has a production process from Monday to Saturday,
however  one is interested only in  the specific control production process   that runs on tueday and and friday.
Or a professor that goes through an illness and the school is interested in the number of days without  his math lectures that take place only in Monday, Wednesday and Friday.
As an example of the stretchy  cluster,  consider the example of the replicas of an earthquake.
The days with replicas not necessarily occur every day until  the movement of the tectonic plate stops. It
may occur days with an days without replicas, but all of them are consequence of the same. 
One would like  to consider all the days with replicas (not all the days) within a unique cluster.
The number of days between two subsequent  days with replicas  could vary. 
In that sense, a further example is in  statistics of gene occurrence.
Suppose one is interested in the gene {\tt ACGAATACG}. 
Once it is found in the sample, the occurrence of that gene once again immediately  (overlapped in {\tt ACG} ) is favored.  Our results apply then to estimate the samples that can overestimate the number of occurrences of that gene. \\

Since the publication of the fundamental Chernoff method,  much of the work  to obtain a good inequality relays on dealing with the moment generating function (MGF) of the variables.
As we mentioned above, this was overcome  imposing extra conditions on the variables.
In our result we do not ask for such restrictions.
The novelty of our argument for the  proof of the deviation relays on four fundamental ingredients acting together:
(1) We construct a recursive operator based on the clustered structure of  of occurrences of the phenomena.
(2) A spectral decomposition of that recursive operator.
(3)  A partial interpolation of the MGF of the unbounded cluster:  quadratic interpolation (Lemma \ref{quartic}) and  linear interpolation (Lemma \ref{cubic}), together with  quadratic  interpolation of the indicator of a cluster.
(4) Optimization of the region of interpolation.

Our first result, Theorem \ref{teoQ}, provides the general form  of the upper bound for the deviation.
It is of theoretical character, in the sense that an analytic explicit expression envolves the solution of the roots 
of a quartic polynomial.  
However it is feasible to implement it numerically.
Our second resut, Theorem \ref{teoR}, presents several explicit bounds that allow  interpretation   with respect to the parameters of the model. We still observe the dichotomy mentioned  at the beggining of this introduction 
when presenting   Bernstein'inequality.

The paper is organized as follows. In Chapter 2 we present the framework. Chapter 3 contains both, the main results and examples. Chapter 4 is devoted to the proofs. Chapter 5 is the appendix that collects classical results.

\section{Definitions}

\subsection{Framework}

We consider an ergodic  stationary stochastic  process $(X_n)_{n \in \mathbb{N}}$ over a finite or countable  alphabet  $\mathcal{A}$.
The space $\mathcal{A^\mathbb N}$ is endowed with the $\sigma$-algebra generated by finite strings and with the stationary measure $\P$. 
The shift to the left  $\sigma: \mathcal{A^\mathbb N}\to  \mathcal{A^\mathbb N}$ is the function that for any $w\in  \mathcal{A^\mathbb N}, w=(w_j)_{j\in \mathbb{N}}$
acts  as $(\sigma(w))_j= w_{j+1}$ for all $j\in \mathbb{N}$.

Let $A\in\mathcal{F}_0^{k-1}$ be a positive measure  set,  defined on the sub-$\sigma$-algebra generated by $X_0,\dots,X_{k-1}$.
A cluster of occurrences of $A$ is a group of close observations of the event  $A$ along the sample $X_0,\dots.X_{n-1}$.  To formalise.  close will mean a distance up to $ q \in\Nset$ from one occurrence  of $A$, up to the next. 
In or work, the allowed maximum distance $q$ between consecutive occurrences of $A$ is  arbitrary and fixed. 

 To simplify notation put   $I_j= \mathbbm{1}_{  \sigma^{-j}(A) }$ for the indicator function of the event $A$ happening at time $j\in\mathbb N$.
 Thus we say that there is a cluster of $A$ between times $m$ and $m'$ if the following conditions hold simultaneously:
 \begin{enumerate}
 \item $I_m=1$ and $I_{m'}=1$,
\item $I_j=0$ for all ${m-q} \le j \le  {m-1}$ and all ${m'+1} \le j \le  {m'+q}$,
\item for every $j_0$ such that $m \le j_0 \le m'$ with  $I_{j_0}=1,$ there exists (at least one)  $j_1$ with $|j_0-j_1|\le q$ such that  $I_{j_1}=1$.  
\end{enumerate}
The counting function of the number of occurrences of $A$  from time  (or space) $m$ to $m'\ge m$  in the sample $X_0^{n-1}$ is given by
\begin{equation} \label{defYn}
N_{m}^{m'} =   \sum_{j=m}^{m'} I_j \ .
\end{equation}
Specifically, we  write $N_{m}=N_0^{m-1}$ for the counting over $m$ observations starting from the origin. \\

\subsection{The problem}

\emph{Suppose $I_0=1$,} say  an occurrence of $A$ is observed at the beginning of the sample.
The end of this first cluster in the sample is defined as the variable
\[
\tau = \inf \left\{  j_0 \in \mathbb{N} | \   I_{j_0}=1, \  \cap_{j=j_0+1}^{j_0+q }  \{ I_j=0 \} \right\} .
\]
Thus $N_1^\tau$  is the number of further occurrences of $A$ in the first observed cluster.
The large deviation property depends  on the behaviour of the moment generating function of the 
variable $N_1^\tau$. Since it is defined  for $I_0=1$, we call it the \emph{(reminig) cluster size}. Thus, we set
\[
M(\lambda) = \E(e^{\lambda N_1^\tau} \ | \ I_0=1 ) \ . 
\]

We assume the following  property, that control certain correlation.
Fix $q\in \Nset$.   
Define
 \[
 \Psi(q) := \sup_{r,s }   \frac{ \P \left(   N_{\tau+q+1}^{n-1}=s \  | \   I_0=1 ,   N_1^{\tau}=r  \right)  }{ 
                                             \P \left(   N_{\tau+q+1}^{n-1}=s \right) }    \ .
 \]
 We say that $X_0^{n-1}$ is a $\Psi(q)$-process  if $\Psi(q)<\infty$.

This is a kind of   $\psi(q)$ coefficient  in the classical $\psi$-mixing condition but much weaker in several senses. 
Firstly, the definition depends only on the functions $\tau, N_{\tau+g+1}^{n-1}$.
Thus it depends on the sub-$\sigma$-algebra generated by $A$ only, not on the full $\sigma$-algebra, as is traditionally defined $\psi$-mixing. 
This allows to apply our results to processes that have \emph{non uniform} decay of correlation, namely unbounded slower decay, depending on the region of the phase-space,  as is the case of the House of Cards Markov chain, renewal processes, regeneration processes   and the Manneville-Pomeau map \cite{asmussen, cox, gallo, manneville}.
Further, 
it is a one-side inequality, and just a single coefficient,  of  a kind of what is call $\psi^+$-mixing in the literature, 
We do not ask for the sequence of constants $\Psi(j), j\ge 0,$ to converge to one (cf. \cite{bradley, doukhan}). 
Of course, the closer, the better the result.

\section{Main results}

\subsection{Theorems}

Our main result will be established in terms of  $\mu_1$, the mean residual cluster size defined as
\[
\mu_1=\E(N_1^\tau | I_0=1) \ , 
\]
and 
the  following further    parameters.
For any $\lambda_0>0$ define
\[
 \theta_0=\theta(\lambda_0)=  \frac{e^{\lambda_0} -1 -\lambda_0}{\lambda_0^2} \ ,  \qquad  \kappa_0=\kappa(\lambda_0)=\frac{\E(e^{\lambda_0 N_1^\tau}  | I_0=1)-1  -\lambda_0 \mu_1 }{\lambda_0^2 } \ . 
\]
We emphasise that both parameters are increasing functions of $\lambda_0$ and diverge as  it does.
Yet,  $\theta_0$ is finite for every choice of $\lambda_0$. However, if the moment generating function of the residual size of the cluster $M(\lambda)$ has a singularity at, say $\lambda_s$, then $\kappa_0$ has a singularity there, too.
Furthermore, a first order approximation gives 
$1/2 < \theta_0 \approx 1/2 + \lambda_0/6$ and $\mu_2/2 < \kappa_0 \approx \mu_2/2 + \lambda_0\mu_3/6$, 
where $\mu_2=\E( (N_1^\tau)^2 | I_0=1)$ and $\mu_3=\E( (N_1^\tau)^3 | I_0=1)$ are respectively the second and third moment of the residual cluster size.\\



We put $p=\P(A)$, and fix $ 0<\epsilon < 1-p$.
 Our main result reads

\begin{teo} \label{teoQ} Let $(X_n)_{n\in\mathbb{N}}$ be a stationary $\Psi$-process with invariant measure $\P$.
Let $q\in \Nset$ and set  
$p^*=\Psi(q) p$.
Set $\lambda_0= \epsilon / p^*(1+2\mu_1)$ and consider the quartic polynomial 
\[
Q(\lambda) =  -  \epsilon  \lambda  +  p^* [   ( \theta_0  + \mu_1)   \lambda^2  +  (\theta_0\mu_1+\kappa_0)   \lambda^3  +   \theta_0 \kappa_0  \lambda^4  ] \ .
\]
 Let $\lambda_{\rm Q} =\arg\min_{0 < \lambda \le \lambda_0} Q(\lambda) $.
Then
\begin{equation} \label{main}
\P(N_n \ge (p+\epsilon) n ) \le C_0 e^{    Q(\lambda_{\rm Q}) n} +  \sum_{i=1}^{q-1} C_i z_i^{n}   e^{ -(p+\epsilon)n} \ .
\end{equation}
Here $z_i,  i=1,\dots,q-1$ are the secondary roots of the polynomial $a_p(z)=z^q-z^{q-1}-d$ 
with  constant $d=p^* (e^{\lambda_{\rm Q}}-1)M(\lambda_{\rm Q})$. 
Yet $C_i,  i=0,\dots,q-1$  are constants (with respect to $n$).
\end{teo}

 \begin{remark}
 Explicit expressions (eq. (\ref{const1=})  and (\ref{const2=})) for the constants $C_i$'s  
and bounds  (eq. (\ref{const1bound0}), (\ref{const1boundj}) and (\ref{const2bound}))
will be given in the proof of the theorem. We put them outside the theorem for the clarity of the exposition.
They are the solutions of the equations system $VC=I$ where $V$ is the Vandermonde matrix given by the roots of the polynomial $a_p(z)$  
 and $I$ is the vector  of initial conditions   $I=(\E(e^{N_0}),\E(e^{N_1}),...,\E(e^{N_{q-1}}) )$. 
We will  further show that this solution can be  written as the sum of two vector $C'$ and $C''$. One. $C'$, is close to the vector $(1,0,...,0)$,
and $C''$   depends on the distance between  the vector  $I$, and the corresponding one computing the  expectations with respect to the product measure, $I'= ( 1,  \E(e^{N_1}),...,(\E(e^{N_1}))^{q-1} )$.
\end{remark}

\vskip0.5cm
Extensive analytical expression for $\lambda_Q$  is available through
Cardano-Tartaglia's formula for the roots of a cubic polynomial. We state it in the appendix for completeness. 
Even when not transparent at a glimpse,  it is numerically feasible,  after computing the several parameters defining $Q$.
The next  theorem provides more transparent inequalities. We  introduce  another  constant that plays the role of $\kappa_0$ in Theorem \ref{teoQ}.
Let
$$\kappa_{1} = \E\left(  \frac{e^{\lambda_0 N_1^\tau}-1}{\lambda_0} | I_0=1 \right) .$$

As  before, we note the first order approximation  $\kappa_1 \approx \mu_1+ \lambda_0\mu_2/2$.

\begin{teo} \label{teoR} Let $(X_n)_{n\in\mathbb{N}}$ be a stationary $\Psi$-process with invariant measure $\P$.
Set $\lambda_0= \epsilon / p^*(1+2\mu_1 )$. Consider the cubic polynomial 
\[
R(\lambda) =   -\epsilon \lambda + p^* [ ( \theta_0 + \kappa_1)  \lambda^2 +  \theta_0 \kappa_1   \lambda^3 ] \  .
\]
 Let $\lambda_{\rm R}  =\arg\min_{0 < \lambda \le \lambda_0} R(\lambda) $.
Then
\begin{eqnarray}
Q(\lambda_{\rm Q} )  &\le&    R(\lambda_{\rm R} )  \nn \\
&\le& - \frac{ {\epsilon}^2 }{2 [  \sqrt{ {p^*}^2(\theta_0 +\kappa_1)^2 +3 \theta_0 \kappa_1\  p^* \epsilon } 
+p^*(\theta_0 +\kappa_1)] }   \nn\\ 
&\le& - \frac{ {\epsilon}^2 }{4   \sqrt{ {p^*}^2(\theta_0 +\kappa_1)^2 +3 \theta_0 \kappa_1\  p^* \epsilon }  }   \label{explicit1} \\ 
&\le& - \frac{ {\epsilon}^2 }{4p^*(\theta_0+\kappa_1)  \sqrt{ 1 + 3\epsilon/4p^* } } \ .  \label{apply}
\end{eqnarray}
\end{teo}

 \begin{remark}
 From the last inequality we still identify two regimes of deviations.
 For mild ones ($\epsilon \ll p^*$) the upper boud is of order  $\epsilon^2 /2p^*(\theta_0+\kappa_1) $, that is a sub-Gaussian tail.
 For very large ones  ($\epsilon \gg p^*$) one obtain the bound 
 $ \epsilon^2/{ 2 (\theta_0+\kappa_1)  \sqrt{  p^*\epsilon } }$.
 \end{remark} 

 \begin{remark}
 The explicit expression for  $\lambda_R$ is derived in Lemma \ref{cubgen} and is given by
 $\lambda_R=   \tfrac{\epsilon}{\tilde{b}+b}$, with
 $b= p^* ( \theta_0 + \kappa_1) $ and $\tilde{b}= \sqrt{ [p^* ( \theta_0 + \kappa_1)]^2+3 p^* \theta_0 \kappa_1\epsilon}$.
 Long direct calculations give that
$R(\lambda_R)$ is
 $$
 R(\lambda_R)=
 \frac{  2 b^3  +  9ab\epsilon - 2 \tilde{b}^3    }{  (3a)^3}
 ,$$
 where $a= p^*  \theta_0  \kappa_1$.
 \end{remark}

 \begin{remark}
A main constant  in the last upper bound is given by   $\theta_0+\kappa_1$. 
To understand  its magnitude  one can use  Taylor's expansion, already mentioned right after the definitions 
of $\theta_0$ and $\kappa_1$.
While  $\theta_0=1/2+O(\lambda_0)$, $\kappa_1=\mu_1+O(\lambda_0)$. Thus the magnitude is given by
$1/2 + \mu_1 + O(\lambda_0)$.

 \end{remark}

\subsection{Examples}

\begin{itemize}
\item[(a)] \emph{Markov chain.}
A Markov chain is already a rich source of examples illustrating different situations.
We present one which is a slight perturbation of a product measure.
Informally, over an alphabet $\mathcal{A}$ we put only one symbol for which, once the chain reaches it,
the chain tends to stay in that state. For the remaining ones, the variables are independent.

Formally, let without loss of generality $\mathcal{A}=\{1,\dots,\ell\}$
and let  $(q_i)_{i\in {\mathcal A}}$ be a probability measure  over  $\mathcal{A}$. 
Define the following Markov chain over ${\mathcal A}$ through the following transition probabilities.
For $i\not=\ell$, set $\Q(i,j)=q_j$ for all $j\in \mathcal{A}$.  
Set $\Q(\ell,\ell)=\rho$ and $Q(\ell,j)= q_j \ (1-\rho)/(1-q_\ell)$, for all $j\not=\ell $.
Kac's Lemma gives that the stationary measure of  state $\ell$ is given by 
$p:=[\rho+ (1-\rho)(1+1/q_\ell) ]^{-1}  =  q_\ell/[ 1-\rho +q_\ell ] $.
We consider the occurrences of state $\ell$ in that Markov chain, which  tends to appear in clusters by construction. 
Specifically,   $\E(I_t I_{t+1})-\E(I_t)\E(I_{t+1}) = p(\rho-p)$, which obviously is positive for $\rho>p$. 
We also choose that the cluster of occurrences of $\ell$, ends as soon as the chain exits $\ell$, thus $q=1$.
Direct computations using the markovian property give
$\Psi(1) = \max\{ q_\ell/p , (1-q_\ell)/(1-p) \} = \max\{1-\rho+q_\ell  ,  (1-q_\ell)(1-\rho+q_\ell)/(1-\rho)  \}  $.
The residual size of the cluster is geometric with success probability  $1-\rho$ and thus 
$\mu_1= \rho/(1-\rho)$.
The moment generating function is
\[ \E(e^{\lambda N_\tau}|I_0=1 ) =  \frac{1-\rho}{1-\rho e^{\lambda}}  \ ,
\]
which has a singularity at $\lambda_s = -\ln \rho$. 
Thus, Chernoff inequality can be applied for 
$\lambda < \lambda_s$. 
That means we consider $\lambda_0 < \lambda_s,$ explicitly
$$ \frac{\epsilon }{p^*}\frac{1-\rho}{1+\rho} <     -\ln \rho .
$$
A direct computation gives that  Theorem \ref{teoQ} and Theorem \ref{teoR} apply, at least,  for   $\epsilon \le 2p^*$ (and $p+\epsilon\le1$).  
This condition gives that the square root in the upper bound (\ref{apply}) is bounded from above by  $\sqrt{5/2}$, and obviously from below by 1. 
Furthermore, 
\[
\kappa_1 = \frac{\rho (e^{\lambda_0}-1)   }{ (1-\rho e^{\lambda_0}) \lambda_0} 
=  \frac{\rho  }{ 1-\rho e^{\lambda_0}}        ( {\theta_0 \lambda_0 +1})  \ .
\]
We finally obtain 
$\theta_0+\kappa_1=  1/2 + \rho/(1-\rho e^{\lambda_0}) + O(\lambda_0)$.  
The cluster enlarges as $\rho$ goes to 1, which in turns makes $\kappa_1$, dominant and   diverging.
This makes the exponential decay of the deviation  slower. This puts in evidence the bigger fluctuations, more difficult to detect.


\item[(b)] \emph{Random walk.}  A particular case of the previous example is a symmetric random walk over a complete graph of $\ell$ vertexes and with a  loop only at vertex $\ell$. 
In this case $q_j=1/\ell$, for all $j<\ell.$
We choose $\rho=(\ell-1)/\ell$, and the probability for the walk to go from vertex $\ell$ to any other vertex equals $1/\ell(\ell-1)$.  
For large $\ell$, the walk takes a long time to reach $\ell$, but once it is in, it remains there for a long time.  One gets  $p=1/2,  \ \mu_1=\ell-1,  \  1/(1+2\mu_1)=1/(2\ell-1),$ and $\Psi(1)=2(\ell-1)/\ell.$
This renders 
$$ \lambda_0 = \frac{\epsilon \ell }{  (\ell-1)(2\ell-1)  }  \approx \frac{\epsilon}{2\ell} \ .
$$
Finally, 
$\theta_0+\kappa_1=   1/2 +  \tfrac{\ell-1}{\ell  (1- (\ell-1)e^{\lambda_0}/\ell)}     + O(\epsilon/\ell )$.
For large $\ell$. one has   $\kappa_1\approx \ell$.

\item[(c)] \emph{Markov chain with stretching cluster.}
We now slightly modify the previous example to  present one with $q=2$. It can be further easily generalised to any $q$.
Informally, we add a state linked only to state $\ell$.
Precisely, let  $\mathcal{A}=\{1,\dots,\ell, \ell+1\}$.
Yet, we keep the same   probability measure  $(q_i)_{i\in \{1,\dots,\ell\}} $
and  the transition probabilities:
for $i=1,\dots, \ell-1$, set $\Q(i,j)=q_j$ for all $j\le \ell$. 
Further, set $\Q(\ell,\ell)=\rho_0$ and $\Q(\ell,\ell+1)=\rho_1$ with $\rho_0+\rho_1=\rho$.
As before,   $Q(\ell,j)= q_j \ (1-\rho)/(1-q_\ell)$ for all $j\le \ell$.
To complete the transition probabilities we set $Q(\ell+1,\ell)=1$.
We still consider the occurrences of state $\ell$ in this Markov chain.
Now,  the stationary measure of  $\ell$ is   
$p:=[\rho_0+2 \rho_1+  (1-\rho)(1+1/q_\ell) ]^{-1}  =  q_\ell/ [ 1- \rho + q_\ell + \rho_1 q_\ell ] $.

We choose that a cluster of occurrences of $\ell$   ends as soon as the chain runs  two consecutive times over states different from $\ell$,  namely
$\tau=\inf \{ k \ | \  X_{k+1}\not=\ell, X_{k+2}\not=\ell  \}.$  This means $q=2$.
The coefficient $\Psi(2)$ has the same expression that in the previous example, with the corresponding actual value of $p$.
The residual size of the cluster is geometric, now with success probability  
$1-[\rho+ (1-\rho)q_\ell]=(1-\rho)(1-q_\ell)$.  
This gives 
$$
\mu_1=  \frac { 1-  (1-\rho)(1-q_\ell) }{    (1-\rho)(1-q_\ell)  }, 
$$
and  
$$
\frac{1}{ 1+2\mu_1}    =        \frac{    (1-\rho)(1-q_\ell)  }{ 2-  (1-\rho)(1-q_\ell) } .
$$
Recall that $\lambda_0=\epsilon/p^*(1+2\mu_1)$.
Similarly to the first example 
$\theta_0+\kappa_1=  1/2 + \mu_1+ O(\lambda_0)$. The second term is larger than the first one for 
$(1-\rho)(1-q_\ell)<1/3.$

Furthermore, in this case, the associated polynomial is $z^2-z-d$, which has the roots
$z_0= (1+\sqrt{1+4d})/2,$ and $z_1= (1-\sqrt{1+4d})/2$. 
Recall that  $d=p^*(e^{\lambda_0}-1)M(\lambda_0)$.
The constants $C_0, C_1$ are the solution of
\[
\begin{bmatrix}
    1       &   1     \\
    z_{0}  & z_{1}  \\
\end{bmatrix} 
\begin{bmatrix}
    C_0     \\
    C_1  
\end{bmatrix} 
=
\begin{bmatrix}
    1     \\
    x_1  
\end{bmatrix}  ,
\]
where $x_1=1+p(e^{\lambda_0}-1)$. It has  the solution 
$$
C_0= \frac{x_1-z_0}{z_1-z_0} \ , \qquad C_1= \frac{x_1-z_0}{z_0-z_1} .
$$

\item[(d)]  \emph{Generalized Smith model.}
A  regeneration model is a process constructed concatenating blocks of symbols generated independently and with the same distribution.
In a generalized regeneration model, the distribution of the next block is no longer independent, it  depends (only) on the last block.
The Smith model (see \cite{smith} )is a specific regeneration model, in which each new block is constructed with only one symbol, repeated a random number of times.
The generalized Smith model is a generalized regeneration process
 in which the condition is not to repeat the symbol of the previous block.
The specific construction is the following.
Firstly, there is a distribution $(q_j)_{j\ge 1},$ over the positive integers.
Now, given a positive integer $a$ observed in the previous block, the next symbol is chosen with this distribution \emph{ normalized to do not repeat $a$.} Namely, set $R_n$ as the set of realizations of the process with regeneration 
between  times $n$ and $n+1$.
Then $\P( X_{n+1}=a \ | \ X_n=a, R_n)=0, $ and $\P( X_{n+1}= b \ | \ X_n=a, R_n)=q_b/(1-q_a)$ for all $b\not=a$.
Once $b$ has being  chosen, the length of the block of $b$'s has one of two possible  values:
1 with  probability   $(b-1)/b$,   
or
$b+1$  with  probability   $1/b$.
Loosely speaking,   large values of $b$ usually appear  isolately, but eventually they appear in  a cluster of size $b+1$.
We fix a symbol $b$ and observe  a cluster of occurrences of $b$. We  set that the cluster ends as soon as another symbol different from $b$ appears in the sample. This means $q=1$.
It follows easily  that,  for any $b$, the mean size of each block of $b$'s is 2,
the conditional mean of the residual block size, given that one symbol  $b$ is observed, is 
$\mu_1=1/2,$ and the moment generating function of the residual cluster is
$$
M(\lambda)= \frac{b-1}{b} + \frac{e^{\lambda (b+1)}-1}{e^{\lambda}-1 } \frac{1}{b(b+1)} ,
$$
defined for all $\lambda>0$. This gives  $\lambda_0=\epsilon/2p^*$ and 
$$
\kappa_1= \frac{1}{b\lambda_0}  \left[    \frac{1}{b+1}. \frac{e^{\lambda_0(b+1)}-1}{e^{\lambda_0}-1 }  -1  \right]  
= \frac{1}{2} + O(\lambda_0) \ .
$$
This gives $\theta_0$ and $\kappa_1$ of the same order.


\end{itemize}

\section{Proofs}

\subsection{Moment generating function}

Before proving  our main result, we need a number of lemmas that  their final goal is to bound 
$M(\lambda)$, the moment generating function of the (residual) cluster size $N_1^\tau$.  
Recall that
$ \theta_0=  \frac{e^{\lambda_0} -1 -\lambda_0}{\lambda_0^2} $ and   
$\kappa_0=\frac{M(\lambda_0) -1 -\lambda_0 \mu_1 }{\lambda_0^2 } \ . $

\begin{lema} \label{quartic}
Fix $\lambda_0>0$.
Then, for all $ \lambda \in [0, \lambda_0]$ the following inequalities hold
\begin{itemize}
\item[(a)] $e^{\lambda}-1 \le \lambda +  \theta_0 \lambda^2 \ ,  $
\item[(b)]  $  M(\lambda)  \le  1 + \mu_1 \lambda  + \kappa_0  \lambda^2  \ .   $
\end{itemize}
\end{lema} 


\begin{proof}
By convexity,  one gets 
\[  e^{\lambda}-1 \le \lambda + \frac{e^{\lambda_0} -1 -\lambda_0}{\lambda_0^2} \ \lambda^2 \  ,  \]
which is the first inequality by definition of $\theta_0$.
Similarly
\[  e^{\lambda  N_1^\tau}  \le  1 + \lambda  N_1^\tau + \frac{e^{\lambda_0  N_1^\tau}  -1 -\lambda_0  N_1^\tau }{\lambda_0^2 }  \lambda^2  \ .   \]
Taking expectation the second  inequality follows.
\end{proof}

\begin{remark}
Taylor's expansion gives immediately  inequalities in the opposite direction
\begin{itemize}
\item[(a)] $e^{\lambda}-1 \ge \lambda +  \frac{1}{2} \lambda^2 \ ,  $
\item[(b)]  $  M(\lambda)  \ge  1 + \mu_1 \lambda    \ .   $
\end{itemize}
It folluws that no better constants are expected to be obtained for the polynomials $Q$ and $R$ in  Theorem \ref{teoQ} and Theorem \ref{teoR} respectively.
\end{remark}

The previous lemma are used below  to provide a quartic polynomial that bounds  $\E(e^{\lambda N_1^\tau}|I_0=1)$. 
Less accurate but more transparent results  come from  cubic bounds, that  rise up from the following lemma.
 Recall that $\kappa_{1} = \E(  \frac{e^{\lambda_0 N_1^\tau}-1}{\lambda_0} | I_0=1) .$

\begin{lema} \label{cubic}
Fix $\lambda_0>0$. For all $\lambda \in [0, \lambda_0]$ the following inequality holds

$M(\lambda) \le  1+  \kappa_1\lambda  \ .$
  \end{lema}

\begin{proof} 
A  bit of algebra gives
\begin{eqnarray}
M(\lambda)
  &=& 1+  \lambda  \E(  \frac{e^{\lambda N_1^\tau }-1}{\lambda} | I_0=1)  \\
  &\le& 1+  \lambda  \E(  \frac{e^{\lambda_0 N_1^\tau}-1}{\lambda_0} | I_0=1)  . \\
  \end{eqnarray}
  
This ends the proof.
\end{proof}

 \subsection{Large deviations}

We are now ready to prove our main results.

\begin{proof}[Proof of Theorem \ref{teoQ}]
Chernoff bound gives
\begin{align*}
 \P( N_n \ge  \alpha n )  & \le \inf_{\lambda>0}  \ e^{-\lambda\alpha n}  \E(e^{\lambda N_n} ) . 
\end{align*}

 %
 %
 %
 
 
 Conditioning on $I_0$ one gets
\begin{eqnarray}
 \E(e^{\lambda N_n} ) 
 &=& pe^{\lambda}\E(e^{\lambda N_1^{n-1}} | I_0=1) + (1-p)\E(e^{\lambda N_1^{n-1}} | I_0=0)  \nonumber  \\
 &=&     \E(e^{\lambda N_1^{n-1}} ) + p(e^{\lambda}-1)\E(e^{\lambda N_1^{n-1}} | I_0=1)  .   \label{xn}\\
 \end{eqnarray}
 
 To have $I_0=1$ means that a cluster  already began  at the beginning of the sample.  
So, to compute the  expectation in the last display, we write $N_1^{n-1}$ as the sum of the remaining occurrences within this first cluster (namely the residual time of the cluster) and the sum of the  occurrences  of $A$ after finished this first cluster. That is
\[
N_1^{n-1} =  N_1^{\tau} +  N_{\tau+q+1}^{n-1} .
\]
By   hypothesis 
\begin{eqnarray*}
\E(e^{\lambda N_1^{n-1}} | I_0=1 ) 
&=&   \E \left(  e^{\lambda ( N_1^{\tau} +  N_{\tau+q+1}^{n-1})}  \  | \   I_0=1    \right) \\
&\le&  \E(e^{\lambda N_1^{\tau} }|I_0=1 ) \Psi(q) \E( e^{\lambda N_{\tau+q+1}^{n-1}} ) \\
&=& M(\lambda) \Psi(q) \E( e^{\lambda N_{\tau+q+1}^{n-1}} ) . 
\end{eqnarray*}
Since  $N_j^{n-1}$ is monotonic on $j$, one has the trivial bound   $N^{n-1}_{\tau+q+1} \le N_{q+1}^{n-1}$. 
Therefore, still using stationarity, the last factor is bounded from above by
$ \E( e^{\lambda N_0^{n-q}} ). $
To simplify notation,  set  $x_n= \E(e^{\lambda N_ {n}})$. Recall also that $p^*=\Psi(q) p$.
One gets the recursive inequality
\begin{eqnarray*}
x_n 
&\le&   x_{n-1}  +  p^* (e^{\lambda}-1)M(\lambda) x_{n-q}   . \\
\end{eqnarray*}
Its associated (complex) polynomial  is
\[
a_p(z)=z^q  -  z^{q-1}  -    F(\lambda) \ ,  
\]
with  $F(\lambda)=p^* (e^{\lambda}-1)M(\lambda)$.
This polynomial has  $q$ different roots. The method of recursive equations gives that
 \begin{equation} \label{geral}
 x_n \le  \sum_{i=0}^{q-1}  C_i z^n_i  \ ,
 \end{equation}
 where $C_i, i=0,\dots,q-1$ are given by the initial conditions  $x_0,,...,x_{q-1}$ and  $z_0,...,z_{q-1}$ are the roots of $a_p$.
 The rest of the proof consists  in two problems: to control the roots and to control the constants. \\
 
 {\sl The roots.} 
 To control the roots one may use calculus and Rouch\'e's Theorem. 
 Notice that $a_p(z)=x^{q-1}(x-1)-F(\lambda)$.
The roots of $a_p$   are thus continuous perturbations of the roots of $z^q  -  z^{q-1} = z^{q-1}(z-1)$. These are: 1,  a single root and 0,  a root with multiplicity $q-1$.
Exactly one root, say   $z_0$ (the perturbation of the root 1 of the homogeneous part of $a_p$), is real and positive.  Further, Rouch\'e's Theorem gives directly that the moduli of \emph{all} the roots are bounded from above by $1+F(\lambda)$.
 (Actually, the  moduli of $z_1,...,z_{q-1}$  are smaller than $|z_0|$ and require a more detailed inspection for a sharper bound. The bounds  for the remaining roots of a polynomial of the form $a_p$ are obtained independently in Lemma \ref{sharproots} in the appendix with further bounds in Lemma \ref{secondary} and Lemma \ref{lower}). 
 One concludes that the leading term for the upper bound in (\ref{geral})  is given by
 \begin{equation} \label{inf}
 \left[  \inf_{\lambda>0} e^{-\lambda \alpha}  (1+ F(\lambda)  ) \right]^ n    .
 \end{equation}
 


A  basic inequality  gives $1+F \le e^F$. 
It follows that the infimum in the brackets in   (\ref{inf}) is bounded by
\[
\inf_{ \lambda > 0}  e^{-\lambda \alpha + F(\lambda) } \ .
\]
Now one needs to minimize the exponent of the above display, that is $-\alpha \lambda + F(\lambda)$.
Firstly fix $\lambda_0>0$.
 It worth noticing that this function is convex, and at $\lambda=0$ equals zero and with derivative equal to $-\epsilon$. Thus it has a minimum at some positive $\lambda$. To bound this minimum, we fix $\lambda_0>0$.
 Lemma \ref{quartic} gives  that, for all $\lambda\in [0 , \lambda_0]$, 
 \[
 -\alpha \lambda + F(\lambda) \le
  -  \epsilon \lambda +   p^* \left[ ( \theta_0  + \mu_1)   \lambda^2 +  (\theta_0\mu_1+\kappa_0)   \lambda^3   +  \theta_0 \kappa_0  \lambda^4 \right]  
  =: Q(\lambda)  .
\]
Similarly, for $\lambda\ge 0$, 
one has   $Q'(0)<0$ and   $Q'$  increasing and diverging. So $Q$ has a unique  minimum 
$\lambda_{\rm Q}$. 
We need to check that  it is smaller or equal than $\lambda_0$.
Direct computations show that
$Q'( \epsilon/2p^*(\theta_0+\mu_1)  )$ is 
positive, so $ \epsilon/2p^*(\theta_0+\mu_1)$ is an upper bound for $\lambda_Q$. 
Still, $\theta_0 >1/2$. 
Finally
\[
\lambda_{\rm Q} <  
\frac{\epsilon}{2p^*(\frac{1}{2} +\mu_1)}  \ . 
\]



 Therefore, for the choice  $\lambda_0= \epsilon / p^*(1+2\mu_1)$, 
 one has the   guarantee that  $\lambda_{\rm Q} <\lambda_0$.  
This ends the  proof of inequality (\ref{main}). 

\vskip0.5cm

{\sl The constants.}
 To compute the constants $C_i, i=0,\dots, {q-1}$ one has to solve the linear system $VC=I$ where $C=(C_0,...C_{q-1})$ is the vector of constants, $I=(x_0,...x_{q-1})$ is the vector of initial conditions $x_i= \E(e^{\lambda N_ {i-1}}),$ and $V$ is the Vandermonde matrix generated by the roots $(z_0,...z_{q-1})$,
Namely 
\[ V=
\begin{bmatrix}
    1       &   1      &   1     & \dots & 1 \\
    z_{0}  & z_{1} & z_{2} & \dots & z_{q-1} \\
    \hdotsfor{5} \\
    z_{0}^{q-1}       & z_{1}^{q-1} & z_{2}^{q-1} & \dots & z_{q-1}^{q-1}
\end{bmatrix} \ .
\]
Consider the similar linear system $VC'=I'$ where $I'=(1,x_1,x_1^2,...,x_1^{q-1})$. 
Therefore $C=C'+ C''$, with $C'= V^{-1}I' $ and $C''= V^{-1}(I-I')$. 
We treat now both terms.
Cramer's rule gives that $C'_j=\det(V'_j)/\det(V)$ where $V'_j$ is the matrix $V$ changing the column $j$ for the vector $I'$. Since the components of $I'$ are powers of $x_1$,  one gets that $V'_j$
 is also a Vandermonde matrix, specifically, generated by $z_0,...,z_{j-1}, x_1, z_{j+1},...,z_{q-1}$. The classical expression for the determinant of a Vandermonde matrix (see, for instance, \cite{gautschi}) gives
 \[
 C'_j = \prod_{i\not=j} \frac{x_1 - z_i }{z_j - z_i } \ .
 \]
 Notice firstly that $x_1>1,$ and that $a_p$ has $z_0$ as  unique  positive root, thus $x_1$
is different from $z_i, \ i=1,...,q-1$. 
Further  $x_1= 1+p(e^{\lambda}-1) < p^*(e^{\lambda}-1)M(\lambda)$,
 thus strictly different from $z_0$.  Therefore the  numerator can be rewritten as
 \[
 \frac{a_p(x_1)}{x_1 - z_j}  \ .
 \]
For the denominator one proceeds as in the numerator, but since $z_j$ is a root, it can be computed taking limit
\[
\lim_{z\to z_j} \frac{a_p(z)}{z-z_j} = {a_p'(z_j)} \ .
\]
One concludes that
\begin{equation} \label{const1=}
C'_j
=  \frac{a_p(x_1)}{ (x_1-z_j) \ a_p'(z_j)}  
=  \frac{ a'_p(\zeta)(x_1-z_0)  }{   a_p'(z_j)  \ (x_1-z_j) }  \ ,
 \end{equation}
 for some $\zeta \in  (x_1, z_0)$. 
 Heuristically, $x_1$ is close to the root $z_0$, resulting in $C'_0$ close to one and the remaining constants  close to zero.
 Formally,  
 $C'_0 = a'_p(\zeta)/a'_p(z_0)$. 
Further,  $a'_p(x)>0$ at least for $x\ge 1$, it follows that  
\begin{equation} \label{const1bound0}
\frac{a'_p(x_1) }{a'_p(z_0)}  < C'_0 < 1.
\end{equation}
Consider  now   $C'_j, j=1,\dots,q-1$.  Notice that this is for $q\ge 2$. Otherwise this set of  constants is empty.
In this case, one gets the bounds
\begin{equation} \label{const1boundj}
  \frac{a'_p(x_1) }{a_p'(z_j)  \ (x_1-z_j) }  \ (z_0-x_1)  \le
 |C'_j| \le    \frac{a'_p(z_0) }{a_p'(z_j)  \ (x_1-z_j) }  \ (z_0-x_1)  \ . 
\end{equation}
Since $x_1$ is close to $z_0$, the modulus of the constants are close to zero, depending on the spectral gap of the roots of $a_p$.

Now consider the second term $C''$. 
Set $\delta_j := |x_j-x_1^j|  .$  Notice that $\delta_0=\delta_1=0$.
As before, $C_j''=  \det(V_j(\delta))/\det(V)$, where $V_j(\delta)$ is the matrix $V$ replacing the $j$-column by the $\delta_j$'s.
Thus
\begin{equation} \label{const2=}
C_j''=  \frac{\det(V_j) }{a'_p(z_j)}  ,
\end{equation}
and
\begin{equation} \label{const2bound}
|C_j''|= \max_j \delta_j \   ||V^{-1}||  ,
\end{equation}
The first factor must be bounded by \emph{ad hoc} properties of the process and the second one with Corollary \ref{VandermondeCor}

\end{proof}

\vskip1cm

\begin{proof}[Proof of Theorem \ref{teoR}]
  By an application of  Lemma \ref{cubic}   and the first inequality in Lemma \ref{quartic}
one gets  
\begin{eqnarray*}
 Q(\lambda)  \le  -\epsilon \lambda +  p^*\left[ (\theta_0+\kappa_1) \lambda^2 +{\theta_0 \kappa_1 }  \lambda^3  \right] , 
 \end{eqnarray*}
 for all $\lambda\in [0,\lambda_0]$ and any fixed $\lambda_0$.
 First inequality follows. The second one is a consequence of an application of Lemma \ref{cubgen}(b) in the appendix.
Further, basic calculus gives $\theta_0 \kappa_1 \le   (\theta_0+\kappa_1)^2/4. $
Third inequality now is immediate. 
Still, by Lemma \ref{cubgen} (a), the minimum of $R(\lambda)$ is reached at a point bounded from above by 
$\epsilon/ 2p^*(\theta_0+\kappa_1)$ which in turn, by the lower bounds of $\theta_0$ and $\kappa_1$, is bounded from above by 
$\epsilon/ p^*(1+2\mu_1)=\lambda_0$.
  This  ends   the proof.
\end{proof}

\section{Appendix}

\subsection{Roots of $z^q-z^{q-1}-d$}

In this section we provide information for the roots of a polynomial of the form $z^q-z^{q-1}-d$, with $d>0$.  We begin  rewriting it  as 
$z^{q-1}(z-1)-d$. This is a translation of $z^{q-1}(z-1)$ which has a single root at 1 and a  root at 0, of multiplicity $q-1$.
So the roots of the first polynomial are continuous perturbation of these ones. 
We call $z_0$ the root of  $z^q-z^{q-1}-d$ corresponding to  the perturbation of the root 1 of $z^q-z^{q-1}$. 
This root is real. Direct computations give the bound  $ z_0< 1+d$.
Heuristically, the remaining roots are close to the  $q-1$-th roots o of $d$.  We formalise this information in the next lemmas.

\begin{lema} \label{realproot} Consider the polynomial $a_p(z)= z^q-z^{q-1}-d$ with a constant $d>0$. 
There is a real root  that verifies 
\[
1+ \frac{ d }{ (1+d)^{q-1}} < z_0 <1+  d \ .
\]
\end{lema}
\begin{proof}
One has $a_p(1)=-d$  and  $a_p(x) $ is increasing and convex  for (not only)    $x \ge 1$. Further $a'_p(1)=1$. The tangent line of $a_p$ at $x=1$ provides that $a_p(1+d)>0$. The secant between $(1, a_p(1))$ and $(1+d, a_p(1+d))$ gives that $a_p(1+ d/(1+d)^{q-1} )<0 $.
\end{proof}

The lemma below gives  bounds for the second root $z_1$ of the polynomial  $z^q-z^{q-1}-d$.
For second root we mean the (strictly) complex root with smallest positive angle.

\begin{lema} \label{sharproots} Consider the polynomial $a_p(z)= z^q-z^{q-1}-d$ with a constant $d>0$. 
There is a  root $z_1=r_1 e^{i\beta_1}$  with  $\pi/(q-1)<\beta_1 \le 2\pi/q$ and 
$$r_1 \le f^{-1}(d) \le  \left( \frac{d}{\sin(\pi/(q-1))}    \right)^{1/(q-1)} \ , $$ 
where $f(r)=r^{q-1}\sqrt{  1+ r^2  -2 r \cos(\pi/(q-1)) }$.
Yet, there is no root with positive angle less than $\pi/(q-1)$.
\end{lema}

\begin{remark}
Notice that the bound provided by Lemma \ref{sharproots} improves the bound $1+d$, 
obtained for  all the roots $z_0, z_1,\dots,z_{q-1}$ in the proof of  Theorem \ref{teoQ}.
Further, the approximation $\sin x \approx x$ says that  the modulus of $z_1$ is  close to $d^{q-1}$.
\end{remark}

\begin{proof}
Consider $z= r e^{i\beta}$ with $r>0,  0 < \beta \le 2\pi/q$.
The roots  of $a_p$ must verify  $(q-1)\beta +\beta_{z-1}=0 \  \pmod{2\pi}$, where $\beta_{z-1}$ is the angle of $z-1$. 
Now fix a radius $r>0$.  and run $\beta$ between  0 and  $2\pi/q$.
For $\beta$ going to zero, $\beta_{z-1}$ approches to $\pi$, $\pi/2$ and zero, as long as $r<1, r=1$ and $r>1$, respectively.
In any case,  for $0< \beta\le \pi/(q-1)$,  one has   $\beta_{z-1} < \pi$, which yields 
$0<(q-1)\beta +\beta_{z-1} <2\pi$. So there are no roots for that range of $\beta$.
Further,  one has  $\beta_{z-1}>\beta.$ 
As $\beta$ goes to $2\pi/q$ one gets  $(q-1)\beta +\beta_{z-1} > 2\pi$. 
One concludes that there exists one  root with angle between  $\pi/(q-1)$ and $2\pi/q$. \\

We now  consider any  root of $a_p$ with angle $0<\beta\le\pi$.
 and find bounds for the radius $r$.
 The strategy is the following.
 For that fixed positive angle $\beta$, one gets that  $|z^{q-1}(z-1)|$,  at $z=0$ is equal to 0. Further, this  modulus, as a function of $r$, is not  increasing on $r$ but it is bounded from below by  an increasing and diverging function of $r$.  
Thus larger than $d$ for large enough $r$. A continuity argument let us conclude that there must me a solution for
$|z^{q-1}(z-1)|=d$.  An increasing  lower bound of the modulus of that root follows. 
Basic geometry gives
\begin{eqnarray*}
|z-1|
 &=&   \sqrt{  (r\sin\beta)^2 +  (1-r\cos\theta)^2  }  \\
&=&   \sqrt{  1+ r^2  -2 r \cos\beta  }  \\
&\ge& \sqrt{  1+ r^2  -2 r \cos(\pi/(q-1)) } . \\
\end{eqnarray*}
Last inequality follows since $\beta > \pi/(q-1)$. Basic calculus gives that  $f(r)$, as defined in the lemma,  
is increasing and diverges, thus greater  than $d$ for $|z|$  larger than   $f^{-1}(d)$.  
To get an explicit bound for $|z_1|$ one can compute the minimum of the last display, which is  $\sin(\pi/(q-1))$. 
Thus $|z^{q-1}(z-1)|\ge r^{q-1}\sin(\pi/(q-1))$ which is larger than $d$  for $r\ge   (d/\sin(\pi/(q-1)))^{1/(q-1)}$.
This ends the proof.
 \end{proof}


 The next lemma says that for the  roots of  $a_p$, the bigger the angle, the smaller the modulus.
 
 \begin{lema} \label{secondary}
Suppose $z'=r' e^{i\beta'}$ and $z"=r" e^{i\beta"}$ are roots of $a_p= z^q-z^{q-1}-d$  with $0\le \beta' < \beta"\le \pi$. 
Then $r"<r'$.
 \end{lema}
 \begin{proof}
 Any root must verify $|z^{q-1}(z-1)|=d$. Fix $r>0$ and consider the circumference $|z|=r$. The previous equation becomes
 $r^{q-1} |z-1|=d$. For such circumference $r^{q-1}$ is fixed and $|z-1|$ increases as the angle of $z$ runs from 0 to $\pi$.
 Thus for the  roots of $a_p$, the smaller is the angle the larger is its modulus.
 This ends the proof.
  \end{proof}

The two lemmas above implies  that the bound for $|z_1|$ given by Lemma \ref{sharproots} is  a bound for  the moduli of the roots $z_1,...,z_{q-1}$. \\

The next lemma provides a lower bound for the moduli of the roots.

\begin{lema} \label{lower}
Let $z$ be  a root  of $a_p= z^q-z^{q-1}-d$ with $|z|=r$. Then
\[
r \ge \left\{
\begin{array}{lr}
(d/2)^{1/(q-1)} & {\rm  \ for \ } d\le 1 \\
(d/2)^{1/q} & {\rm  \ for \ } d\ge 1 
\end{array}
.  \right.
\]
\end{lema}

\begin{proof}
Following the proof of Lemma \ref{sharproots} 
\begin{eqnarray*}
|z-1|
&=&   \sqrt{  1+ r^2  -2 r \cos\beta  }  \\
&\le& \sqrt{  1+ r^2  -2 r \cos(\pi) }  \\
&=& r+1 . \\
\end{eqnarray*}
It follows that for $r\le1$  one gets $|z^{q-1}(z-1)| \le r^{q-1} 2$,
and  for $r\ge1$, one gets $|z^{q-1}(z-1)| \le 2r^{q} $.
In consequence, for $d\le 1$ there are no roots with $r\le(d/2)^{1/(q-1)}$, and for $d\ge 1$
  there are no roots with
$r\le(d/2)^{1/q}$.
This ends the proof.
\end{proof}

\subsection{Quartic and cubic polynomials}

{\sl The quartic minimum.}
To find a local minimum of a quartic polynomial one can use Cardano's formula to find the roots of its derivative, a cubic polynomial.

\begin{lema} \label{cardano}
The roots of a cubic polynomial  $x^3+a x^2+b x+c$ are given by  $x_{\rm root} = y_{\rm root} - a/3,$
where $y_{\rm root}$ are the roots of the reduced form $y^3 + \p y + \q $ with
\begin{eqnarray*}
\p &=&  b -\frac{a^2}{3}  \ ,  \\
\q &=&    2(\frac{a}{3})^3  - \frac{ab}{3} + c \  .
\end{eqnarray*}
 The roots are reached at
\[
y_{\rm root} =   \left[  - \frac{\q}{2} -  \left(  (\frac{\q}{2})^2 +   (\frac{\p}{3})^3  \right)^{1/2}    \right]^{1/3} +   
             \left[ -   \frac{\q}{2} +  \left(  (\frac{\q}{2})^2 +   (\frac{\p}{3})^3   \right)^{1/2}   \right]^{1/3}  .
\]
\end{lema} 

{\sl The cubic minimum.}
The next lemma  presents   bounds for both, the local minimum and its position, of an  homogeneous cubic  polynomial.

 \begin{lema} \label{cubgen}
 Let $e(x)=a x^3+b x^2-\delta x$, with $a, b$ and $\delta$ positives. 
 Let $x_{\min}$ be the point of the local mimimum of this polynomial.
 Then
 \begin{itemize}
 \item[(a)]  ${\delta}/{2\sqrt{b^2+3a\delta}}  \le  x_{\min} \le  \delta/2b$,
\item[(b)]  $  -\frac{ \delta^2 }{ \sqrt{b^2 + 3a\delta} +b}  \le 
 e(x_{\min})  \le -\frac{ \delta^2 }{2( \sqrt{b^2 + 3a\delta} +b)} $.
 \end{itemize}
 \end{lema} 
 
 \begin{proof}
  The largest root of the derivative of a polynomial of the form $e(x)=a x^3+b x^2-\delta x$, with $a, b$ and $\delta$ positives, which gives the position of the local minimum of $e(x)$ is    
 \begin{equation} \label{exact}
 x_{\min}=   \frac{\sqrt{ 4b^2+12a\delta}-2b}{6a}  = \frac{\sqrt{b^2+3 a\delta}-b}{3a} .
 \end{equation}  
A direct computation  gives 
  \[
 x_{\min} =  \frac{\delta}{\sqrt{   b^2 + 3a\delta} +b } \ .
\]
Item (a) follows immediately. \\


 
  Now, notice that
  \[
   e(x) 
  =  x(e'(x)-2ax^2-bx)  \ .
  \]
  Since $x_{\min}$ is a root of $e'$
  \[e(x_{\min}) = -x^2_{\min} \ (2ax_{\min}+b)  . \]
  By (\ref{exact})
\[   2ax_{\min}+b     =  \frac{2\sqrt{   b^2 + 3a\delta} +b}{3}  \ ,
 \]  
which is  upper and lower bounded by $\sqrt{b^2 + 3a\delta} +b$, and its half, respectively.
One concludes that 
 \[
 -\frac{ \delta^2  }{ \sqrt{b^2 + 3a\delta} +b } \le    e(x_{\min}) \le -\frac{ \delta^2 }{2( \sqrt{b^2 + 3a\delta} +b)} \ . 
   \]

  This ends the proof.

  \end{proof}

 \subsection{Vandermonde matrix}
  
  It is well known that the stability of the norm of the inverse of a Vadnermonde matrix depenqs on the distance between each pair of
coefficients that   generates  the matrix.
 We quote here a result  due to  Gautschi \cite{gautschi}  for easy reference for the reader.
  
  Set the norm of a $q\times q$ matrix  $A=(a_{ji}), \ 1\le  j.i\le q$ by
  $$
  ||A|| =  \max_{1\le j\le q} \sum_{i=1}^{q} |a_{ji}| \ .
  $$
  \begin{teo} \label{Vandermonde}
  Let $V=V(z_0,\dots, z_{q_1})$ be a Vandermonde with all the $z_j$ differents.
  Then,  
\[
||V^{-1}|| \le \max_{0\le j\le q-1} \prod_{i\not=j}  \frac{1+|z_i|}{ |z_j-z_i| } \ .
\]
 \end{teo}
 
 We derive a corollary to bound the norm  inverse of the Vandermonde matrix in the specific case of Theorem\ref{teoQ} and  Theorem \ref{teoR}.
 
 \begin{cor}  \label{VandermondeCor}

 Let $V$ be a Vandermonde matrix with entries given by the roots of the polynomial $a_p(z)=z^q-z^{q-1}-d,$ with  positive integer $q$ and positive $d$. Then 
 \[
 ||V^{-1}|| \le  (1+r_1)^{q-2}  \max_{1 \le j \le q-1} \left\{  \frac{1+r_1  }{   |a'_p( z_{j})|}, \frac{ 2+d  }{   |a'_p( z_0) |  } \right\} \ ,
 \]
 where $r_1$ is the modulus of the root with the smallest positive angle.
  \end{cor}
\begin{proof}
This follows directly using the argument in the analysis of the constants in the proof of Theorem \ref{teoQ}
and the bounds in Lemma \ref{realproot} and Lemma \ref{sharproots}.


\end{proof}

\section*{Acknowledgements}
This article was produced as part of the activities of  FAPESP  research projects 
 2023/13073-8  "Modelagem de Sistemas Estoc\'asticos".




\begin{thebibliography}{99}

\bibitem{asmussen}
Asmussen, S.
\textit{Applied Probability and Queues}.  
2nd ed., Springer, New York, 2003. 

\bibitem{azuma}
Azuma, K.  
\textit{Weighted sums of certain dependent random variables.}  
Tôhoku Mathematical Journal, \textbf{19} (1967), 357–367.

\bibitem{bennett}
Bennett, G.  
\textit{Probability inequalities for the sum of independent random variables.}  
J. of the Am. Stat. Ass. \textbf{57} (1962), 33–45.



\bibitem{bernstein}
Bernstein, S.
\textit{On a modification of Chebyshev's inequality and of the error formula of Laplace.}
Ann. Sci. Inst. Sav. Ukraine, Sect. Math., \textbf{1} (1924), 38–49.


\bibitem{bradley}
Bradley, R. C.
\textit{Basic Properties of Strong Mixing Conditions. Part I}, 
Probability Surveys, vol.~2, pp.~107–144, 2005.

\bibitem{chernoff}
Chernoff, H.  
\textit{A measure of asymptotic efficiency for tests of a hypothesis based on the sum of observations.}  
Ann. of Math. Stat. \textbf{23} (1952), 493–507.

\bibitem{cox}
Cox, D. R.
\textit{Renewal Theory.} 
Methuen \& Co. Ltd., London, 1962.

\bibitem{cramer}
Cramér, H.  
\textit{Sur un nouveau théorème-limite de la théorie des probabilités.}  
Actualités Scientifiques et Industrielles, No. 736. Hermann, Paris, 1938.
\bibitem{doukhan}
Doukhan, P.
\textit{Mixing: Properties and Examples}. Lecture Notes in Statistics, vol.~85, Springer-Verlag, New York, 1994.



\bibitem{feller}  
Feller, W.  
\textit{Generalization of a probability inequality of Tchebycheff.}  
Ann. of Math. Stat. \textbf{14} (1943), 389–400.
\bibitem{gallo}
Gallo, S.
\textit{Chains with unbounded variable length memory.}
Stoc. Proc. and their App. \textbf{121}  (2011),  2889--2915.

\textit{Chains with unbounded variable length memory}, Stochastic Processes and their Applications, 2011.


\bibitem{gautschi}
Gautschi,  W.
\textit{On inverses of Vandermonde and confluent Vandermonde matrices,}
Numerische Mathematik, \textbf{4} (1962) 117--123. 

\bibitem{hoeffding}
Hoeffding, W.  
\textit{Probability inequalities for sums of bounded random variables.}  
J/ of the Am. Stat. Ass. \textbf{58} (1963), 13–30.

\bibitem{janson}
Janson, S.  
\textit{Large deviations for sums of partly dependent random variables.}  
Random Structures \& Algorithms \textbf{24} (2004), 234–248.

\bibitem{janson2}
Janson, S.  
\textit{Large deviation inequalities for sums of indicator variables.}
Preprint \url{https://arxiv.org/abs/1609.00533?utm_source=chatgpt.com}

\bibitem{lezaud}
Lezaud, P.  
\textit{Chernoff-type bound for finite Markov chains.}  
Ann. of App. Prob. \textbf{8} (1998), 849–867.

\bibitem{lu}
Lu, D., Song, L., and Zhang, T.  
\textit{Large deviations for sum of UEND and $\varphi$-mixing random variables with heavy tails.}  
Statistics: A Journal of Theoretical and Applied Statistics, \textbf{50}(7) (2016), 1512–1525.

\bibitem{manneville}
Manneville,  P. and Pomeau, Y.
\textit{Intermittent transition to turbulence in dissipative dynamical systems}.
Comm. Math. Phys.  \textbf{74} (1980),  189–197.

\bibitem{merlevede}
Merlevède, F., Peligrad, M. and Rio, E.  
\textit{Bernstein inequality and moderate deviations under strong mixing conditions.}  
High Dimensional Probability V: The Luminy Volume (E. Giné, V. Koltchinskii, W. Li, J. Zinn, eds.),  
Institute of Mathematical Statistics Collections, \textbf{5} (2009), 273–292.

\bibitem{mikosch}
Mikosch, T. and Rodionov, I.  
\textit{Precise large deviations for dependent subexponential variables.}  
Stochastic Processes and their Applications, 

\bibitem{okamoto}
Okamoto, M.  
\textit{Some inequalities relating to the partial sum of independent random variables.}  
Ann. of the Inst.of Stat. Math. \textbf{10} (1958), 29–35.

\bibitem{paulin}
Paulin, D.
\textit{Concentration inequalities for Markov chains by Marton couplings and spectral methods.}
Electron. J. Probab. \textbf{20} (2015), 1–32.

\bibitem{smith}
Smith, R.\,L.  
\textit{A counterexample concerning the extremal index.}  
Advances in Applied Probability, \textbf{20} (1988), 681–683.

\end{thebibliography}
  \end{document}